\documentclass[10pt, reqno]{amsart}

\usepackage{amsmath, amssymb,amsthm}
\usepackage{hyperref}

\usepackage{latexsym}
\usepackage{fancyhdr}
\usepackage{pdflscape}
\usepackage{longtable}
\usepackage{rotating}
\usepackage{verbatim}
\usepackage{tikz}
\usepackage{subfigure}
\usepackage{mathrsfs}
\usepackage{cleveref}
\usepackage{enumitem}
\makeatletter
\newcommand{\mylabel}[2]{#2\def\@currentlabel{#2}\label{#1}}
\makeatother

\usepackage{mdwlist}
\usepackage{color}
\usepackage{tensor}

\usepackage[all]{xy}
\newdir{ >}{{}*!/-10pt/@{>}}
\newdir^{ (}{{}*!/-5pt/@^{(}}
\newdir_{ (}{{}*!/-5pt/@_{(}}

\newcounter{mnotecount}[section]

\newcommand{\rmnote}[1]{}

\theoremstyle{plain}
\newtheorem*{theorem*}{Theorem}
\newtheorem{theorem}{Theorem}[section]
\newtheorem*{lemma*}{Lemma}
\newtheorem{lemma}[theorem]{Lemma}
\newtheorem*{proposition*}{Proposition}

\newtheorem*{corollary*}{Corollary}

\newtheorem*{claim*}{Claim}

\newtheorem*{conjecture*}{Conjecture}

\newtheorem*{question*}{Question}
\theoremstyle{definition}
\newtheorem{remark}[theorem]{Remark}
\newtheorem*{remark*}{Remark}
\newtheorem*{definition*}{Definition}

\newtheorem*{example*}{Example}

\def\A{\;\forall}
\def\E{\;\exists}
\def\p{\partial}
\def\o{\circ}
\def\preceq{\preccurlyeq}

\def\Im{\on{Im}}

\let\on=\operatorname

\newcommand{\ol}{\overline}
\newcommand{\ul}{\underline}

\def\db{\ol \partial}

\def\al{\alpha}
\def\be{\beta}
\def\ga{\gamma}
\def\de{\delta}
\def\ep{\varepsilon}
\def\ve{\varepsilon}
\def\ze{\zeta}

\def\la{\lambda}

\def\rh{\rho}

\def\si{\sigma}

\def\ph{\varphi}
\def\vh{\varphi}

\def\om{\omega}

\def\Ga{\Gamma}

\def\La{\Lambda}

\def\Om{\Omega}

\def\C{\mathbb{C}}

\def\N{\mathbb{N}}
\def\Q{\mathbb{Q}}
\def\R{\mathbb{R}}

\def\cB{\mathcal{B}}

\def\cE{\mathcal{E}}

\def\cH{\mathcal{H}}

\def\cL{\mathcal{L}}

\def\cP{\mathcal{P}}

\def\cS{\mathcal{S}}

\def\fM{\mathfrak{M}}

\def\fS{\mathfrak{S}}

\def\sD{\mathscr{D}}

\title{Nonlinear conditions for ultradifferentiability}

\author[D.N.~Nenning, A.~Rainer, and G.~Schindl]{David Nicolas Nenning, Armin Rainer, and Gerhard Schindl}

\address{Fakult\"at f\"ur Mathematik, Universit\"at Wien, Oskar-Morgenstern-Platz~1, A-1090 Wien, Austria.}
\email{david.nicolas.nenning@univie.ac.at}
\email{armin.rainer@univie.ac.at}
\email{gerhard.schindl@univie.ac.at}

\begin{document}

	\begin{abstract}
		A remarkable theorem of Joris states that a function $f$ is $C^\infty$ if
		two relatively prime powers of $f$ are $C^\infty$.
		Recently, Thilliez showed that an analogous theorem holds in Denjoy--Carleman classes
		of Roumieu type.
		We prove that a division property, equivalent to Joris's result,  
		is valid in a wide variety of ultradifferentiable classes. 
		Generally speaking, it holds in all dimensions for non-quasianalytic classes. 
		In the quasianalytic case we have general validity in dimension one, but 
		we also get validity in all dimensions for certain quasianalytic classes.    
	\end{abstract}

\thanks{AR was supported by FWF-Project P 32905-N, DNN and GS by FWF-Project P 33417-N}

\maketitle

	\section{Introduction}

	A remarkable theorem of Joris \cite[Th\'eor\`eme 2]{Joris82} states:
	if $f:\R \rightarrow \R$ is a function and $p,q$ are relatively prime positive integers, then
	\begin{equation}\label{eq:smooth}
		f^p,f^q \in C^\infty \implies f \in C^\infty.
	\end{equation}
	Since smoothness can be tested along smooth curves by a theorem of
	Boman \cite{Boman67}, one immediately infers that the implication \eqref{eq:smooth} holds on arbitrary
	open subsets of $\R^d$, $d \ge 1$, and on smooth manifolds.
	(On the other hand, the regularity of a \emph{single} power of a function generally 
	says nothing about the regularity of the function itself; 
	e.g.\ $(\mathbf{1}_{\Q} - \mathbf{1}_{\R \backslash \Q})^2 = 1$,
	where $\mathbf{1}_A$ is the indicator function of a set $A$.)

	It was soon realized that the statement also holds for complex valued functions 
	and it led to the study of so-called pseudoimmersions 
	\cite{DuncanKrantzParks85,JorisPreissmann87,JorisPreissmann90,Rainer18}.
	A simple proof based on ring theory was given by \cite{AmemiyaMasuda89}.

	Only recently Thilliez \cite{Thilliez:2020ac} showed that Joris's result carries over to
	Denjoy--Carleman classes of Roumieu type $\cE^{\{M\}}$.
	These are ultradifferentiable classes of smooth functions defined by certain growth properties imposed
	upon the sequence of iterated derivatives in terms of a weight sequence $M$
	(which in view of the Cauchy estimates measures the deviation from analyticity).

	By extracting the essence of Thilliez's proof, we show in this paper that 
	a broad variety of ultradifferentiable classes has a division property equivalent to Joris's result.
	Let $\cS$ be a subring (with multiplicative identity) of the ring of germs at $0 \in \R^d$ of complex valued $C^\infty$-functions. 
	We say that $\cS$ has the \emph{division property $(\sD)$} if for any function germ $f$ at $0 \in \R^d$ 
	we have
	\begin{equation} \label{eq:D}
		\big(j \in \N_{\ge1},~   f^j,f^{j+1} \in \cS\big) \implies f \in \cS.		
	\end{equation}	
	If $\cS$ has property $(\sD)$, then Joris's theorem holds in $\cS$. 
	Indeed, suppose that $p_1,p_2$ are relatively prime positive integers 
	and $f^{p_1},f^{p_2} \in \cS$. All integers $j \ge p_1p_2$ can be written $j=a_1p_1 + a_2 p_2$ for $a_1,a_2 \in \N$,  
	see \cite[p.270]{Joris82}.
	Hence $f^j \in \cS$ for all $j \ge p_1p_2$. 
	Since two consecutive integers are relatively prime, 
	also the converse holds.\footnote{In an earlier version of the paper we   
	considered the division property
	$\big(j \in \N_{\ge1},~  g,q,fg \in \cS,~ f^j= qg\big) \implies f \in \cS$ which resulted from our wish to prove an 
	ultradifferentiable version of the division theorem \cite[Theorem 1]{JorisPreissmann90}. 
	But this property is equivalent to \eqref{eq:D}.  
	}

	\subsection{Results}
	Let us give an overview of our results.  

	The rings of germs in \emph{one dimension} $d=1$ of the following ultradifferentiable classes have property $(\sD)$: 
	\begin{itemize}
		\item $\cE^{[M]}$, Denjoy--Carleman class of Roumieu (\Cref{thm:ThilliezJoris,thm:ThilliezJorisQ})
		and Beurling type (\Cref{thm:ThilliezJorisBQ}),
		\item $\cE^{[\om]}$, Braun--Meise--Taylor classes of Roumieu and Beurling type
		(\Cref{thm:weightfunction}),
		\item $\cE^{[\fM]}$, ultradifferentiable classes defined by weight matrices
		of Roumieu and Beurling type
		(\Cref{thm:matrix1dimJoris}).
	\end{itemize}
	It is understood that certain minimal regularity properties of the weights are assumed (see \Cref{table}) 
	which in particular guarantee that the sets of germs are indeed rings.
	(Note that by convention $[\cdot]$ stands for $\{\cdot\}$, i.e., Roumieu, as well as $(\cdot)$, i.e., Beurling.)

	Interestingly, the proof in one dimension works for quasianalytic and non-quasianalytic classes alike.
	But the tool used to reduce the multidimensional to the one-dimensional statement
	is only available in the non-quasianalytic Roumieu case (\cite{KMRc}, \cite{Schindl14}).
	The (multidimensional) Beurling case can often be reduced to the corresponding Roumieu case.
	Hence we obtain the following multidimensional non-quasianalytic results. 
	The rings of germs in \emph{all dimensions} $d$ of the following ultradifferentiable classes have property $(\sD)$:
	\begin{itemize}
		\item $\cE^{[M]}$, non-quasianalytic Denjoy--Carleman class of Roumieu (\Cref{thm:ThilliezJoris})
		and Beurling type (\Cref{thm:sequenceJoris}),
		\item $\cE^{[\om]}$, non-quasianalytic Braun--Meise--Taylor classes of Roumieu and Beurling type
		(\Cref{thm:functionJoris}),
		\item $\cE^{\{\fM\}}$, non-quasianalytic ultradifferentiable classes defined by weight matrices
		of Roumieu type
		(\Cref{thm:matrixmultdimJoris}).
	\end{itemize}

	For \emph{quasianalytic} Denjoy--Carleman classes of Roumieu type $\cE^{\{M\}}$ in \emph{one} dimension
	the implication \eqref{eq:D} follows from the stronger result, due to Thilliez \cite{Thilliez10}, that
	$C^\infty$-solutions of a polynomial equation
	\begin{equation}
		z^n + a_1 z^{n-1} + \cdots + a_{n-1} z + a_n = 0,
	\end{equation}
	where the coefficients $a_j$ are germs at $0 \in \R$ of $\cE^{\{M\}}$-functions,
	are of class $\cE^{\{M\}}$ (under weak assumptions on $M$).
	This is false for non-quasianalytic classes.
	But it seems to be unknown whether, in the presence of quasianalyticity,
	it holds in higher dimensions.
	In fact, quasianalytic ultradifferentiability cannot be tested on quasianalytic curves (or lower dimensional plots) 
	even if the function in question is known to be smooth (\cite{Jaffe16,Rainer:2019aa}). 

	Hence we think that it is interesting that, 
	combining our proof with a description of certain quasianalytic classes $\cE^{\{M\}}$
	as an intersection of suitable non-quasianalytic ones (due to \cite{KMRq}),
	we obtain that these quasianalytic classes have property $(\sD)$ in \emph{all} dimensions (see \Cref{thm:quasiJoris} and 
	also \Cref{rem:intersectableOmega,rem:intersectableF}).

	Since all considered regularity classes are local, the results for germs immediately give 
	corresponding results for functions on open sets.

	\subsection{Summary of the results}
	
	We list in \Cref{table} the ultradifferentiable rings of germs known to have property $(\sD)$,
    together with the needed assumptions on the weights and the respective references.
    All germs are function germs at $0$ in $\R^d$ for some dimension $d$. The dimension 
    is added as a left subscript, e.g., $\tensor[_d]{\cE}{^{[M]}}$ denotes the ring of germs at $0 \in \R^d$ 
    of $\cE^{[M]}$-functions.
	All notions will be defined below.

	The Roumieu parts of the results in the first and the fifth row are due to Thilliez \cite{Thilliez:2020ac}; see \Cref{,sec:DCR,sec:DCB}.

	\begin{table}[ht]
	\renewcommand\arraystretch{1.1}
	\caption{Ultradifferentiable rings of germs having property $(\sD)$}\label{table}
	\begin{tabular}{c|l|l}
		\multicolumn{2}{l|}{\textbf{Quasianalytic ring of germs} $\vphantom{\dfrac{1}{2}}$}   & \textbf{Reference} \\
		\hline
		$\tensor[_1]{\cE}{^{[M]}}$ &
		\begin{tabular}{l}
			derivation closed 
			\\
			$m$ log-convex \\
			$m_k^{1/k} \to \infty$ (Beurling)
		\end{tabular}	
		&
		\Cref{thm:ThilliezJorisQ,thm:ThilliezJorisBQ}
		\\
		\hline
		$\tensor[_1]{\cE}{^{[\om]}}$ $\vphantom{\dfrac{1}{2}}$ & 
		\begin{tabular}{l}
			$\om$  concave
		\end{tabular}
		 & Theorem \ref{thm:weightfunction}\\
		\hline
		$\tensor[_1]{\cE}{^{[\fM]}}$ &
		\begin{tabular}{l}
			[regular]\\
			{[moderate growth]}
		\end{tabular}
		 & Theorem \ref{thm:matrix1dimJoris}\\
		\hline
		$\tensor[_d]{\cE}{^{\{M\}}}$ &
		\begin{tabular}{l}
			intersectable 
			\\
		 	moderate growth 
		\end{tabular}
	  & Theorem \ref{thm:quasiJoris} \\
		\hline
		\multicolumn{2}{l|}{\textbf{Non-quasianalytic ring of germs} $\vphantom{\dfrac{1}{2}}$}   & \textbf{Reference} \\
		\hline
		$\tensor[_d]{\cE}{^{[M]}}$ &
		\begin{tabular}{l}
			moderate growth 
			\\
			$m$ log-convex
		\end{tabular}
		&
		\Cref{thm:ThilliezJoris,thm:sequenceJoris}	
		\\
		\hline
		$\tensor[_d]{\cE}{^{[\om]}}$ $\vphantom{\dfrac{1}{2}}$ & 
		\begin{tabular}{l}
			$\om$  concave
		\end{tabular}
		 & Theorem \ref{thm:functionJoris}\\
		\hline
		$\tensor[_1]{\cE}{^{[\fM]}}$ &
		\begin{tabular}{l}
			[regular]\\
			{[moderate growth]}
		\end{tabular}
		& Theorem \ref{thm:matrix1dimJoris}\\
		\hline
		$\tensor[_d]{\cE}{^{\{\fM\}}}$ &
		\begin{tabular}{l}
			$\{\text{regular}\}$ \\
		 	$\{\text{moderate growth}\}$ 
		\end{tabular}
		& Theorem \ref{thm:matrixmultdimJoris}\\
	\end{tabular}
\end{table}

We remark that non-quasianalytic Denjoy--Carleman classes $\cE^{\{M\}}$, 
where the weight sequence $M$ lacks moderate growth, do not have property $(\sD)$ in general;
see \cite[Remark 2.2.3]{Thilliez:2020ac}.
The moderate growth condition is rather restrictive (e.g., it implies that the class $\cE^{\{M\}}$ 
is contained in a Gevrey class).
The consideration of the classes $\cE^{[\om]}$ and $\cE^{[\fM]}$ allows to overcome this restriction in the sense
that the implication \eqref{eq:D} holds under weaker moderate growth conditions.

	\subsection{Strategy of the proof} \label{sec:strategy}

	Thilliez's proof of Joris's theorem for $\cE^{\{M\}}$ consists of the following two steps:
	\begin{enumerate}[label=\textnormal{(\roman*)}]
		\item The class $\cE^{\{M\}}$ admits a description by holomorphic approximation
		which is based on a result of Dynkin \cite{Dynkin80} on almost analytic extensions
		and a related $\ol\p$-problem.
		\item If $f^j$, $f^{j+1}$ are of class $\cE^{\{M\}}$ and $g_\ep$, $h_\ep$ are respective holomorphic approximations,
		the quotient $h_\ep/g_\ep$ is a naive candidate for a holomorphic approximation of $f$.
		In order to avoid small divisors one considers
		\[
			u_\ep = \vh_\ep \frac{\ol g_\ep h_\ep}{\max\{|g_\ep|,r_\ep\}^2},
		\]
		where $\vh_\ep$ is a suitable cutoff function and $r_\ep>0$.
		For good choices of $r_\ep$ the function $u_\ep$ has uniform bounds and is close to $f$.
		The solution of a $\ol\p$-problem is used to modify $u_\ep$ in order to obtain a holomorphic approximation of $f$.
		By step (i) we may conclude that $f$ belongs to $\cE^{\{M\}}$.
	\end{enumerate}

	Following the same strategy, we will work with weight matrices $\fM$, since they provide a framework for ultradifferentiability (\Cref{sec:matrix})
	which encompasses Denjoy--Carleman classes (\Cref{sec:sequence}) and Braun--Meise--Taylor classes (\Cref{sec:functions}).
	In \Cref{sec:holapprox} we prove a general characterization result by holomorphic approximation for $\cE^{[\fM]}$
	(\Cref{prop:332v}) which extends step (i);
	it builds on the description by almost analytic extension presented in our recent paper \cite{FurdosNenningRainer}.
	Then we execute a version of step (ii) under a quite minimal set of assumptions, see \Cref{lem:corelemma}. 
	It enables us to easily deduce the main results in \Cref{sec:proof}.

	\section{Denjoy--Carleman classes have property \texorpdfstring{$(\sD)$}{(D)}}
	\label{sec:sequence}
	
	\subsection{Weight sequences and Denjoy--Carleman classes} \label{sec:weightsequences}
	Let $\mu=(\mu_k)$ be a positive increasing (i.e.\ $\mu_k \le \mu_{k+1}$) sequence with $\mu_0=1$.
	We define a sequence $M$ by setting $M_k:= \mu_1 \cdots \mu_k$, $M_0 := 1$,
	and a sequence $m$ by $m_k:=\frac{M_k}{k!}$. Clearly, $\mu$ uniquely determines $M$ and $m$, and vice versa.
	In analogy we shall use sequences $N \leftrightarrow n \leftrightarrow \nu$,
	$L \leftrightarrow \ell \leftrightarrow \la$, etc.

	That $\mu$ is increasing means that $M$ is \emph{log-convex},
	i.e., $\log M$ is convex or, equivalently, $M_k^2\le M_{k-1}M_{k+1}$ for all $k$.
	If in addition $M_k^{1/k} \rightarrow \infty$, we say that $M$ is a \emph{weight sequence}.\par

	Sometimes we will make the stronger assumption that $m$ is log-convex.

	For $\si > 0$ and open $U \subseteq \R^d$, one defines the Banach space
	\[
	\cB^M_\si(U)
	:=\Big\{ f \in C^\infty(U):~ \|f\|_{\si,U}^M:= \sup_{x \in U,\, \al \in \N^d}  \frac{|\partial^\al f (x)|}{\si^{|\al|}M_{|\al|}}<\infty \Big\}
	\]
	and the (local) \emph{Denjoy--Carleman classes of Roumieu type}
	\[
	\cE^{\{M\}}(U):=  \on{proj}_{V \Subset U } \on{ind}_{\si > 0} \cB^{M}_\si(V).
	\]
	For later reference we also consider the global class $\cB^{\{M\}}(V) := \on{ind}_{\si > 0} \cB^{M}_\si(V)$. 	
	Replacing the existential quantifier for $\si$ by a universal quantifier, we
	find the \emph{Denjoy--Carleman classes of Beurling type}
	\[
	\cE^{(M)}(U):=  \on{proj}_{V \Subset U } \on{proj}_{\si > 0} \cB^{M}_\si(V)
	\]
	and $\cB^{(M)}(V) := \on{proj}_{\si > 0} \cB^{M}_\si(V)$.
	We use the notation $\cE^{[M]}$ for both $\cE^{\{M\}}$ and $\cE^{(M)}$, similarly for $\cB^{[M]}$, etc. 

	For positive sequences $M,N$, we write
	$M \preceq N$ if $\sup_{k \in \N} \big(\frac{M_k}{N_k}\big)^{1/k}<\infty$
	and $M \lhd N$ if $\lim_{k \rightarrow \infty} \big(\frac{M_k}{N_k}\big)^{1/k} = 0$.
	We have (cf. \cite[Proposition 2.12]{RainerSchindl14})
	\begin{align*}
		M \preceq N \quad &\Leftrightarrow \quad \cE^{[M]}(U) \subseteq \cE^{[N]}(U),\\
		M \lhd N \quad &\Leftrightarrow \quad \cE^{\{M\}}(U) \subseteq \cE^{(N)}(U),
	\end{align*}
	where for ``$\Leftarrow$'' one has to assume that $M$ is a weight sequence.
	Note that $\cE^{\{(k!)\}}(U)$ coincides with $C^\om(U)$ and so
	the class of real analytic functions 
	is contained in $\cE^{(M)}$($\subseteq \cE^{\{M\}}$) if and only if $m_k^{1/k} \to \infty$.

	Log-convexity of $M$ implies that
	$\cE^{[M]}(U)$ is closed under pointwise multiplication of functions.
	Additional regularity properties for $M$ endow $\cE^{[M]}(U)$ with additional structure, e.g.,
	log-convexity of $m$ implies closedness under composition of functions.
	A crucial assumption in \cite{Thilliez:2020ac} is \emph{moderate growth} of $M$, which reads as follows
	\begin{equation}
	\label{eq:mg}
	\exists C>0~\forall k,j \in \N: ~M_{k+j} \le C^{k+j}M_kM_j.
	\end{equation}
	It implies \emph{derivation closedness}
	\begin{equation}
	\label{eq:dc}
	\exists C>0~\forall k \in \N: ~M_{k+1}\le C^{k+1}M_k.
	\end{equation}
	The last property we need to mention is \emph{non-quasianalyticity} of $M$, that is
	\begin{equation}
	\label{eq:Mnq}
	\sum_{k=1}^\infty \frac{1}{\mu_k} < \infty, 
	\quad \text{ or equivalently }\quad \sum_{k=1}^\infty \frac{1}{M_k^{1/k}} < \infty.
	\end{equation}
	By the Denjoy--Carleman theorem, this condition is equivalent to the existence of
	non-trivial functions with compact support in $\cE^{[M]}(U)$. 
	It is well-known that non-quasianalyticity implies $m_k^{1/k}\rightarrow\infty$.

	Let $\tensor[_d]{\cE}{^{[M]}}$ denote the \emph{ring of germs} at $0 \in \R^d$ of complex valued $\cE^{[M]}$-functions;
	here we assume that $M$ is a weight sequence in order to have a ring.

	\begin{remark}
	There is a slight mismatch
	between our notation (also used in \cite{FurdosNenningRainer})
	and that of \cite{Thilliez:2020ac} (and \cite{RainerSchindl12}).
	We write $M_j=m_jj!$ for weight sequences, so our $m$ corresponds to $M$ in \cite{Thilliez:2020ac}. 		
	\end{remark}

	\subsection{Associated functions}

	Let $m = (m_k)$ be a positive sequence with $m_0=1$ and  $m_k^{1/k} \to \infty$. We define the function
	\begin{equation} \label{h}
		 h_{m}(t) := \inf_{k \in \N} m_k t^k, \quad \text{ for } t > 0, \quad \text{ and } \quad h_{m}(0):=0,
	\end{equation}
	which is is increasing, continuous on $[0,\infty)$, and positive for $t>0$. For large $t$ we have $h_{m}(t) = 1$.
	Furthermore, we need
	\begin{align}
  	\label{counting2}
  	\ol \Ga_{m}(t) &:= \min\{k : h_{m}(t) =  m_k t^k\}, \quad t > 0,
	\end{align}
	and, provided that $m_{k+1}/m_{k} \to \infty$,
	\begin{align}
	\ul \Ga_{m} (t) &:=  \min\Big\{k : \frac{m_{k+1}}{m_k}  \ge \frac{1}{t} \Big\}, \quad t > 0.
	\end{align}
	We trivially have $\ul \Ga_m \le \ol \Ga_m$. 
	If $m$ is log-convex, then $\ol \Ga_{m} = \ul \Ga_{m}$.
	
	We shall use these functions for $m_k = M_k/k!$, where $M$ is a weight sequence satisfying $m_k^{1/k}\to \infty$. 
	Then $m_{k+1}/m_{k}\to \infty$ (since $M_k^{1/k} \le \mu_k$ for all $k$).

	\subsection{Regular weight sequences}
	A weight sequence $M$ is said to be \emph{regular} if $m_k^{1/k} \to \infty$, 
	$M$ is derivation closed, and there exists a constant $C\ge 1$ such that $\ol \Ga_m(Ct) \le \ul \Ga_m(t)$ for all $t>0$.

	\subsection{Denjoy--Carleman classes of Roumieu type have property \texorpdfstring{$(\sD)$}{(D)}} \label{sec:DCR}

	\begin{theorem}[Non-quasianalytic {$\tensor[_d]{\cE}{^{\{M\}}}$}]
		\label{thm:ThilliezJoris}
		Let $M$ be a non-quasianalytic regular weight sequence 
		of moderate growth.
		Then $\tensor[_d]{\cE}{^{\{M\}}}$ has property $(\sD)$.
	\end{theorem}

	This is a special case of \Cref{thm:matrixmultdimJoris} below (cf.\ \Cref{thm:BMTmatrix}). It implies Thilliez's result \cite[Corollary 2.2.5]{Thilliez:2020ac}.

	A quasianalytic one-dimensional version follows from a stronger result in \cite{Thilliez10}: 

	\begin{theorem}[Quasianalytic {$\tensor[_1]{\cE}{^{\{M\}}}$}]
		\label{thm:ThilliezJorisQ}
		Let $M$ be a quasianalytic derivation closed weight sequence such that $m$ is log-convex.
		Then $\tensor[_1]{\cE}{^{\{M\}}}$ has property $(\sD)$.
	\end{theorem}

	\subsection{Denjoy--Carleman classes of Beurling type  have property \texorpdfstring{$(\sD)$}{(D)}} \label{sec:DCB}

	Let us deduce Beurling versions of \Cref{thm:ThilliezJoris,thm:ThilliezJorisQ}.
	We use the following lemma based on \cite[Lemma 6]{Komatsu79b} and \cite[Lemma 7.5]{FurdosNenningRainer}.

	\begin{lemma} \label{lem:redDC}
		   Let $L, M$ be positive sequences satisfying $L \lhd M$. 
   		   Suppose that $m$ is log-convex and satisfies $m_k^{1/k} \to \infty$.
   		   Then there exists a weight sequence $S$ such that $s$ is log-convex, $s_k^{1/k} \to \infty$, and
   		   $L\leq S \lhd M$. 
   		   Additionally, we may assume:
   		   \begin{enumerate}[label=\textnormal{(\roman*)}]
   		   	\item $S$ has moderate growth, if $M$ has moderate growth.
   		   	\item $S$ is derivation closed, if $M$ is derivation closed.
   		   	\item $S$ is non-quasianalytic, if $M$ is non-quasianalytic.
   		   \end{enumerate}  
	\end{lemma}

	\begin{proof}
		Only the supplements (ii) and (iii) were not already proved in \cite[Lemma 7.5]{FurdosNenningRainer}.

		(ii) follows from the fact that
		a weight sequence $M$ is derivation closed if and only if there is a constant $C\ge 1$ such that $M_k \le C^{k^2}$ for all $k$, see \cite{Mandelbrojt52,Matsumoto84}.
		Since $S$ is a weight sequence and $S \lhd M$, also $S$ is derivation closed, by this criterion.

		(iii) It suffices to show that there exists a non-quasianalytic weight sequence $N$ such that $L \le N \lhd M$. 
		Then we apply the lemma to $N \lhd M$ and obtain a weight sequence $S$ with $N \le S \lhd M$ having all 
		desired properties. 

		Let us show the existence of $N$. 
		By $L \lhd M$, we have $\be_k := \sup_{p \ge k} \big( \frac{L_p}{M_p} \big)^{1/p} \searrow 0$.
 		Applying \cite[Lemme 16]{ChaumatChollet94} (see also \cite[Lemma 4.1]{Schmets:2003aa}) 
 		to $\be_k$ and $\al_k = \ga_k := \frac{1}{\mu_k}$, yields an increasing sequence $\de = (\de_k)$ such that
		\begin{gather}
			\label{eq:divergent}
			\de_k \rightarrow \infty,\\
			\label{eq:zerosequence}
			\de_k \be_k \rightarrow 0,\\
			\label{eq:decreasing}
			\frac{\mu_k}{\de_k}  \text{ is increasing},\\
			\label{eq:nqthm}
			\sum_{k=1}^\infty \frac{\de_k}{\mu_k} \le 8 \de_1 \sum_{k=1}^\infty \frac{1}{\mu_k} < \infty.
		\end{gather}
		Then $N_k := \frac{\mu_1 \cdots \mu_k}{\de_1 \cdots \de_k}$
		defines a non-quasianalytic weight sequence, by \eqref{eq:decreasing} and \eqref{eq:nqthm}.
		(Note that $\nu_k = \frac{\mu_k}{\de_k} \to \infty$ is equivalent to $N_k^{1/k} \to \infty$.)
		It satisfies $N \lhd M$ by \eqref{eq:divergent}.
		By \eqref{eq:zerosequence}, there is a constant $C>0$ such that 
		$\de_k \big( \frac{L_k}{M_k} \big)^{1/k} \le C$ for all $k$.
		By the monotonicity of $\de$, this leads to 
		$L_k \le C^k \frac{M_k}{\de_k^k} \le   C^k \frac{M_k}{\de_1 \cdots\de_k} = C^k N_k$.
		After replacing $(N_k)$ by $(C^k N_k)$ we have
		$L \le N \lhd M$.
	\end{proof}

	\begin{theorem}[Non-quasianalytic {$\tensor[_d]{\cE}{^{(M)}}$}]
		\label{thm:sequenceJoris}
		Let $M$ be a non-quasianalytic weight sequence of moderate growth such that $m$ is log-convex.
		Then $\tensor[_d]{\cE}{^{(M)}}$ has property $(\sD)$.
	\end{theorem}
	
	\begin{proof}
		Suppose that $g := f^j, h:= f^{j+1} \in \tensor[_d]{\cE}{^{(M)}}$ for some positive integer $j$. 
		Assume that representatives of these germs are defined in the neighborhood of the closure of 
		some bounded $0$-neighborhood $U$; 
		we denote the representatives by the same symbols. Then the sequence
		\begin{equation} \label{eq:defL}
			L_k:= \max\Big\{ \sup_{|\al|= k, \,x\in U} |g^{(\al)}(x)|, 
			\sup_{|\al|= k, \,x\in U} |h^{(\al)}(x)| \Big\}
		\end{equation}
		satisfies $L \lhd M$. 
		By \Cref{lem:redDC}, there exists a weight sequence $S$ satisfying the assumptions of 
		\Cref{thm:ThilliezJoris} and $L \le S \lhd M$.
		Thus, $f \in \tensor[_d]{\cE}{^{\{S\}}} \subseteq \tensor[_d]{\cE}{^{(M)}}$.
	\end{proof}

	\begin{theorem}[Quasianalytic {$\tensor[_1]{\cE}{^{(M)}}$}]
		\label{thm:ThilliezJorisBQ}
		Let $M$ be a quasianalytic derivation closed weight sequence such that $m$ is log-convex and
		$m_k^{1/k} \to \infty$.
		Then $\tensor[_1]{\cE}{^{(M)}}$ has property $(\sD)$.
	\end{theorem}
	
	\begin{proof}
		This follows from the proof of \Cref{thm:ThilliezJorisQ} in \cite{Thilliez10} (which also works in the Beurling case). 
		Alternatively, we may infer it from \Cref{thm:ThilliezJorisQ} by a reduction argument based on \Cref{lem:redDC} as 
		in the proof of \Cref{thm:sequenceJoris}.
	\end{proof}

	\subsection{A multidimensional quasianalytic result}	

	Let $M$ be a weight sequence and consider the sequence space
	\[
		\La^{\{M\}} := \Big\{ (c_k) \in \C^\N : \E \rh >0 : \sup_{k \in \N} \frac{|c_k|}{\rh^k M_k} < \infty\Big\}.
	\]
	We call a quasianalytic weight sequence $M$ \emph{intersectable} if
	\begin{equation} \label{eq:intersectable}
		\La^{\{M\}} = \bigcap_{N \in \cL(M)} \La^{\{N\}},
	\end{equation}
	where $\cL(M)$ is the collection of all non-quasianalytic weight sequences $N \ge M$ such that $n$ is log-convex.
	The identity \eqref{eq:intersectable} carries over to respective function spaces, since
	$\cB^{\{M\}}(U) = \big\{f \in C^\infty(U) :  (\sup_{x \in U}\|f^{(k)}(x)\|_{L^k_{\on{sym}}}) \in \La^{\{M\}}\big\}$, 
	where $f^{(k)}$ denotes the $k$-th order Fr\'echet derivative and $\|\cdot\|_{L^k_{\on{sym}}}$ the operator norm. 

	Note that a quasianalytic intersectable weight sequence $M$ always satisfies $m_k^{1/k} \to \infty$; an argument is 
	given in \Cref{rem:intersectable} below.

	\begin{theorem}[Quasianalytic {$\tensor[_d]{\cE}{^{\{M\}}}$}]
	\label{thm:quasiJoris}
	Let $M$ be a quasianalytic intersectable weight sequence of moderate growth.
	Then $\tensor[_d]{\cE}{^{\{M\}}}$ has property $(\sD)$.
	\end{theorem}

The proof of this result is given in \Cref{sec:proof}.

	\begin{remark} \label{rem:intersectable}
	In \cite[Theorem 1.6]{KMRq} (inspired by \cite{Boman63}) a sufficient condition for intersectability was given. 
	Let $M$ be a quasianalytic weight sequence with $1 \le M_0<M_1$.
	Consider the sequence $\check M$ defined by 	
	\[
	\check{M}_k:= M_k \prod_{j=1}^k \Big(1 - \frac{1}{M_j^{1/j}}\Big)^k, \quad \check M_0 := 1.
	\]
	If $\check m$ is log-convex, then $M$ is intersectable.

	Not every quasianalytic weight sequence is intersectable, for instance,
	\[
		\La^{\{(k!)\}} \ne \bigcap_{N \in \cL((k!))} \La^{\{N\}} = \La^{\{Q\}}, \quad \text{ where } Q_k = (k\log(k+e))^k); 
	\] 	
	see \cite[Theorem 1.8]{KMRq} and \cite{Rudin62}.
	Every quasianalytic intersectable weight sequence $M$ must satisfy $\La^{\{Q\}} \subseteq \La^{\{M\}}$, 
	and so $m_k^{1/k}$ tends to $\infty$ since clearly $q_k^{1/k}$ does.

	A countable family $\mathbf Q = \{Q^n\}_{n\in \N_{\ge 1}}$ of quasianalytic intersectable weight sequences of moderate growth was constructed in \cite[Theorem 1.9]{KMRq}:
	\[
		Q^n_k = \big(k \log(k) \log(\log(k)) \cdots \log^{[n]}(k)\big)^k, \quad \text{ for } k \ge \exp^{[n]}(1), 
	\]
	where $\log^{[n]}$ denotes the $n$-fold composition of $\log$; analogously for $\exp^{[n]}$.
	
	See also \cite[Section 11]{Schindl14} for a generalization of this concept.
	\end{remark}

	\section{Braun--Meise--Taylor classes have property \texorpdfstring{$(\sD)$}{(D)}}
	\label{sec:functions}

	\subsection{Weight functions and Braun--Meise--Taylor classes}
	A \emph{weight function} is, by definition, a continuous increasing function $\om:[0,\infty) \rightarrow [0,\infty)$
	such that
	\begin{itemize}
		\item[$(\om_1)$] $\om(2t) = O (\om(t))$ as $t \rightarrow \infty$,
		\item[$(\om_2)$] $\om(t) = o(t)$  as $t \rightarrow \infty$,
		\item[$(\om_3)$] $\log(t) = o(\om(t))$  as $t \rightarrow \infty$,
		\item[$(\om_4)$] $t \mapsto \om(e^t)=:\vh_\om(t)$ is convex on $[0,\infty)$.
	\end{itemize}
	One may assume that $\om|_{[0,1]} \equiv 0$ (without changing the associated classes $\cE^{[\om]}$) which we shall tacitly do if convenient.

	Let $U \subseteq \R^d$ be open and $\rho >0$. We associate the Banach space
	\[
	\cB^{\om}_\rho(U):= \Big\{f \in C^\infty(U): 
	\|f\|^\om_{\rho,U}:= \sup_{x \in U,\, \al \in \N^d} \frac{|\partial^\al f (x)|}{e^{\ph_\om^* (\rho|\al|)/\rho}}<\infty \Big\},
	\]
	where $\ph_\om^*(s):= \sup_{t \ge 0} \{st - \ph_\om(t)\}$ is the Young conjugate of $\vh_\om$ (which is finite by $(\om_3)$).
	Then the (local) \emph{Braun--Meise--Taylor class of Roumieu type} is
	\[
	\cE^{\{\om\}}(U):= \on{proj}_{V \Subset U} \on{ind}_{n \in \N} \cB^\om_n(V),
	\]
	and that of \emph{Beurling type} is
	\[
	\cE^{(\om)}(U):= \on{proj}_{V \Subset U} \on{proj}_{n \in \N} \cB^\om_{\frac{1}{n}}(V).
	\]
	Again we use $\cE^{[\om]}$ for $\cE^{\{\om\}}$ and $\cE^{(\om)}$, similarly for $\cB^{[\om]}$ etc.

	For two weight functions $\om$, $\si$ we have (cf. \cite[Corollary 5.17]{RainerSchindl14})
	\begin{align*}
	\si(t) = O(\om(t)) \text{ as }  t\to \infty  \quad &\Leftrightarrow \quad \cE^{[\om]}(U) \subseteq \cE^{[\si]}(U),\\
	\si(t) = o(\om(t)) \text{ as }  t\to \infty \quad &\Leftrightarrow \quad \cE^{\{\om\}}(U) \subseteq \cE^{(\si)}(U).
	\end{align*}
	We say that $\om$ and $\si$ are \emph{equivalent} if they generate the same classes, i.e.,
	$\si(t) = O(\om(t))$ and $\om(t) = O(\si(t))$ as $t \to \infty$.   
	
	A weight function is said to be \emph{non-quasianalytic} if
	\begin{equation}
		\label{eq:nq}
		\int_1^\infty \frac{\om(t)}{t^2}\,dt<\infty.
	\end{equation}
	This is the case if and only if $\cE^{[\om]}(U)$ contains non-trivial functions of compact support (cf.\ \cite{BMT90} or \cite{RainerSchindl12}).

	Let us emphasize that in this paper we treat condition $(\om_2)$ as a general assumption for weight functions; it means that
	the Beurling class $\cE^{(\om)}$ contains the real analytic class. It is automatically satisfied if $\om$ is non-quasianalytic.

	Let $\tensor[_d]{\cE}{^{[\om]}}$ denote the \emph{ring of germs} at $0 \in \R^d$ of complex valued $\cE^{[\om]}$-functions;
	note that $\cE^{[\om]}$ is stable by multiplication of functions for any weight function $\om$.

	\subsection{The associated weight matrix} \label{sec:associatedmatrix}

	Let $\om$ be a weight function.
	Setting $\Om^x_k:= e^{\ph_\om^* (xk)/x}$ defines a weight sequence $\Om^x$ for every $x>0$,
	where $\Om^x \le \Om^y$ if $x \le y$.
	Thus the collection $\mathbf{\Omega}:=\{\Om^x\}_{x>0}$ is a weight matrix (in the sense of \Cref{sec:matrix}).
	Note that $\mathbf{\Omega}$ satisfies a \emph{mixed} moderate growth property, namely
	\begin{equation}
	\label{eq:fctmod}
	\A x>0 \A j,k \in \N:  \Om^x_{j+k} \le \Om^{2x}_j \Om^{2x}_k.
	\end{equation}
	The importance of the associated weight matrix $\mathbf{\Om}$ is that it encodes an equivalent topological description
	of the spaces $\cE^{[\om]}(U)$ as unions or intersections of Denjoy--Carleman classes; see \Cref{thm:BMTmatrix}.
	All this can be found in \cite{RainerSchindl12}.

	\subsection{Braun--Meise--Taylor classes have property \texorpdfstring{$(\sD)$}{(D)}}

	\begin{theorem}[{$\tensor[_1]{\cE}{^{[\om]}}$}]
		\label{thm:weightfunction}
		Let $\om$ be a concave weight function. 
		Then $\tensor[_1]{\cE}{^{[\om]}}$ has property $(\sD)$. 
	\end{theorem}
	
	Evidently, it suffices to assume that $\om$ is equivalent to a concave weight function.

	For the multidimensional analogue we additionally assume non-quasianalyticity.
	
	\begin{theorem}[Non-quasianalytic {$\tensor[_d]{\cE}{^{[\om]}}$}]
	\label{thm:functionJoris}
		Let $\om$ be a non-quasianalytic concave weight function. 
		Then $\tensor[_d]{\cE}{^{[\om]}}$ has property $(\sD)$.
	\end{theorem}	
	
	\Cref{thm:weightfunction,thm:functionJoris} are corollaries of \Cref{thm:matrix1dimJoris,thm:matrixmultdimJoris} below;
	for the proofs see \Cref{sec:proof}.

	\begin{remark} \label{rem:intersectableOmega}
	Every weight sequence $M$ in the family $\mathbf Q$ mentioned at the end of \Cref{rem:intersectable} satisfies 
	\[
		\liminf_{k \to \infty} \frac{\mu_{a k}}{\mu_k} > 1
	\]
	for some positive integer $a$. Hence there is a quasianalytic weight function $\om_M$ (take, e.g., $\om_M(t) := - \log h_M(1/t)$) 
	such that $\tensor[_d]{\cE}{^{[\om_M]}} = \tensor[_d]{\cE}{^{[M]}}$, by \cite[Theorem 14]{BMM07}.	
	So for all $M \in \mathbf Q$, the quasianalytic ring $\tensor[_d]{\cE}{^{\{\om_M\}}}$ has property $(\sD)$. 
	\end{remark}

	\section{The most general version of the theorem}
	\label{sec:matrix}
	
	Let us formulate the main theorems in the most general setting available.
	The conditions we put on abstract weight matrices are tailored in such a way that weight matrices 
	associated with weight functions are contained as special cases.

	\subsection{Weight matrices and ultradifferentiable classes} \label{def:matrix}
	A \emph{weight matrix} $\fM$ is, by definition, a
	family  of weight sequences which is totally ordered with respect to the pointwise order relation on sequences, i.e.,
	\begin{itemize}
		\item $\fM \subseteq \R^\N$,
		\item each $ M \in \fM$ is a weight sequence in the sense of \Cref{sec:weightsequences},
		\item for all $ M, N \in \fM$ we have $ M \le  N$ or $ M \ge  N$.
	\end{itemize}
	
	Let $U \subseteq \R^d$ be open. Given a weight matrix $\fM$,
	we define global classes 
	\begin{align} \label{def:BRM}
		\cB^{\{\fM\}}(U) &:= \on{ind}_{ M \in \fM} \cB^{\{ M\}}(U),
		\\
	 \label{def:BBM}
		\cB^{(\fM)}(U) &:= \on{proj}_{ M \in \fM} \cB^{( M)}(U).
	\end{align}
	The limits in \eqref{def:BRM} and \eqref{def:BBM} can always be assumed countable,
	as is shown in \cite[Lemma 2.5]{FurdosNenningRainer}. Writing $[\cdot]$ for $\{\cdot\}$ and $(\cdot)$, 
	the local classes are defined by
	\[
	\cE^{[\fM]}(U) := \on{proj}_{V \Subset U} \cB^{[\fM]}(V).
	\]

	Let $\tensor[_d]{\cE}{^{[\fM]}}$ denote the \emph{ring of germs} at $0 \in \R^d$ of complex valued $\cE^{[\fM]}$-functions;
	notice that $\cE^{[\fM]}$ is stable by multiplication of functions, since each $M \in \fM$ is a weight sequence.

	\subsection{Regular weight matrices} \label{def:Rregular}
		A weight matrix $\fM$ satisfying
		\begin{itemize}
			\item $m_k^{1/k} \to \infty$ for all $ M \in \fM$ \label{strictInclusion}
		\end{itemize}
		is called $\{\text{\it regular}\}$ or \emph{R-regular} (for Roumieu)  if
		\begin{itemize}
			\item $\forall   M \in \fM \E  N \in \fM \E C\ge 1 \A j \in \N : M_{j+1} \le C^{j+1} N_j$, \label{R-Derivclosed1}
			\item
			$\forall  M \in \fM \E  N \in \fM \E C\ge 1
			\A t>0 : \ol \Ga_{ n}(Ct) \le \ul \Ga_{ m}(t)$, \label{goodR}
		\end{itemize}
		and $(\text{\it regular})$ or \emph{B-regular} (for Beurling) if
		\begin{itemize}
			\item $\forall  M \in \fM \E  N \in \fM \E C \ge 1 \A j \in \N : N_{j+1} \le C^j M_j$, \label{B-Derivclosed1}
			\item $\forall  M \in \fM \E  N \in \fM \E C \ge 1
			\A t>0 : \ol \Ga_{ m}(Ct) \le \ul \Ga_{ n}(t)$. \label{goodB}
		\end{itemize}
		Moreover, $\fM$ is called \emph{regular} if it is both R- and B-regular.
		By our convention, $[\text{regular}]$ stands for $\{\text{regular}\}$ (i.e.\ R-regular) in the Roumieu case and
		$(\text{regular})$ (i.e.\ B-regular) in the Beurling case.

	\subsection{Almost analytic extensions}
	
	Let $h : (0,\infty) \to (0,1]$ be an increasing continuous function which tends to $0$ as $t \to 0$.
	Let $\rh>0$ and let $U \subseteq \R^d$ be a bounded open set.
	We say that a function $f : U \to \C$ admits an \emph{$(h,\rh)$-almost analytic extension} if
	there is a function $F \in C^1_c(\C^d)$ and a constant $C\ge 1$ such that $F|_U = f$ and
	\[
	|\ol \p F(z)| \le C\, h ( \rh d(z,\ol U)), \quad \text{ for } z \in \C^d,
	\]
	where $\ol \p F (z):= \sum_{j=1}^d\frac{\p F(z)}{\p \ol z_j} d\ol z_j $ and $d(z,\ol U) := \inf_{x \in \ol U} |z-x|$ denotes the distance of $z$ to $\ol U$.
	
	Let us apply this definition to the functions $h_m$ from \eqref{h}, where $m_k = M_k/k!$ and $M$ belongs to a given weight matrix $\fM$.
	Let $f : U \to \C$ be a function.
	\begin{itemize}
			\item $f$ is called \emph{$\{\fM\}$-almost analytically extendable} if it has an
			$(h_{ m},\rh)$-almost analytic extension for some $ M \in \fM$ and some $\rh>0$.
			\item $f$ is called \emph{$(\fM)$-almost analytically extendable}
			if, for all $ M \in \fM$ and all $\rh>0$, there is an $(h_{ m},\rh)$-almost analytic extension of $f$.
	\end{itemize}

	\begin{theorem}[{\cite[Corollaries 3.3, 3.5]{FurdosNenningRainer}}]
		\label{thm:roumieuextension}
		Let $\fM$ be a [regular] weight matrix.
		Let $U \subseteq \R^d$ be open.
		Then $f \in \cE^{[\fM]}(U)$ if and only if
		$f|_{V}$ is $[\fM]$-almost analytically extendable for each quasiconvex domain $V$ relatively compact in $U$.
	\end{theorem}

	In \Cref{sec:holapprox} we shall use \cite[Proposition 3.12]{FurdosNenningRainer}, which is a key ingredient of the proof of  
	\Cref{thm:roumieuextension}.

	\subsection{Weight matrices of moderate growth}
	For positive sequences $M$, $N$ set
	\begin{equation}
	\label{eq:rmatrixmg}
	\on{mg}(M,N) := \sup_{j,k \ge 0, \, j+k\ge 1}\left(\frac{M_{j+k}}{N_jN_k}\right)^{1/(j+k)} \in (0,\infty].
	\end{equation}
	We say that a weight matrix $\fM$ has \emph{R-moderate growth} or \emph{$\{\text{moderate growth}\}$} if
	\begin{equation}
		\label{eq:mmg}
		\A M \in\fM \E N \in \fM : \on{mg}(M,N)<\infty,
	\end{equation}
	and \emph{B-moderate growth} or \emph{$(\text{moderate growth})$} if
	\begin{equation}
		\label{eq:mmgb}
		\A M \in\fM \E N \in \fM : \on{mg}(N,M)<\infty.
	\end{equation}
	Again we say that $\fM$ has \emph{moderate growth} if it has R- and B-moderate growth,
	and [moderate growth] stands for $\{\text{moderate growth}\}$ and $(\text{moderate growth})$, respectively.

	\subsection{Denjoy--Carleman and Braun--Meise--Taylor classes in this framework} \label{thm:BMTmatrix}

	By definition, Denjoy--Carleman classes are described by weight matrices $\fM=\{M\}$ 
	consisting of a single weight sequence $M$. 
	Observe that the weight matrix $\fM=\{M\}$ is regular if and only if the weight sequence $M$ is regular, and
	it has moderate growth if and only if $M$ has moderate growth.

	Let $\om$ be a weight function and let $\mathbf \Om$ be the associated weight matrix (cf.\ \Cref{sec:associatedmatrix}).
	Then, by \cite[Corollaries 5.8 and 5.15]{RainerSchindl12}, as locally convex spaces
	\[
		\cE^{[\om]}(U) = \cE^{[\mathbf{\Om}]}(U),
	\]
	and $\cE^{[\om]}(U) = \cE^{[\Om]}(U)$ for all $\Om \in \mathbf \Om$ if and only if
	\[
		\exists H \ge 1 ~\forall t \ge 0 : 2 \om(t) \le \om(Ht) +H,
	\]
	which is in turn equivalent to the fact that some (equivalently each) $\Om \in \mathbf \Om$ has moderate growth; see also \cite{BMM07}.

	The associated weight matrix $\mathbf \Om$ always has moderate growth, by \eqref{eq:fctmod}.
	It is equivalent to a regular weight matrix $\fS$ (that means $\cE^{[\om]} = \cE^{[\fS]}$) if and only if $\om$ is equivalent to a concave weight function.
	In fact, a $\cE^{[\om]}$-version of the almost analytic extension theorem \ref{thm:roumieuextension} holds if and only if $\om$ is
	equivalent to a concave	weight function; see \cite[Theorem 4.8]{FurdosNenningRainer}.
	The weight matrix $\fS = \{S^x\}_{x>0}$ has the property that for each $x>0$ the sequence $s^x$ is log-convex and satisfies
	$\on{mg}(s^x,s^{2x}) =: H < \infty$ and thus $h_{s^x}(t) \le h_{s^{2x}}(Ht)^2$ for all $x,t>0$; see \cite[Proposition 3]{Rainer:2020aa}.

	\subsection{General ultradifferentiable classes have property \texorpdfstring{$(\sD)$}{(D)}}

	\begin{theorem}[{$\tensor[_1]{\cE}{^{[\fM]}}$}]
		\label{thm:matrix1dimJoris}
		Let $\fM$ be a [regular] weight matrix of [moderate growth]. 
		Then $\tensor[_1]{\cE}{^{[\fM]}}$ has property $(\sD)$.
	\end{theorem}

	The proof is given in \Cref{sec:proof}. It builds upon a characterization 
	of the class $\cE^{[\fM]}$ by holomorphic approximation; see \Cref{sec:holapprox}.

	We may infer a multidimensional result, since non-quasianalytic $\cE^{\{\fM\}}$-regularity can be tested along curves;	
	this useful tool is available in a satisfactory manner only in the non-quasianalytic Roumieu setting.
	We need two additional properties of the weight matrix:
	\begin{equation}
	\label{eq:matrixnonqa}
	\E M \in \fM: ~ \sum_{k=0}^\infty \frac{1}{\mu_k}<\infty.
	\end{equation}
	which means that $\cE^{\{\fM\}}$ admits non-trivial	functions of compact support, and
	\begin{equation}
	\label{eq:fdb}
	\A M \in \fM \E N \in \fM: m^\circ \preceq n,
	\end{equation}
	where $m_k^\circ:= \max \{m_j m_{\al_1}\cdots m_{\al_j}: \al_i \in \N_{>0},\, \al_1+\dots+\al_j = k\}$.
	Condition \eqref{eq:fdb} is equivalent to composition closedness of $\cE^{\{\fM\}}$ (which follows from 
	the arguments in \cite[Theorem 4.9]{RainerSchindl12}) 
	and is satisfied by every R-regular weight matrix. 
	Indeed, if $\fM$ is R-regular, then $\cE^{\{\fM\}}$ has a description by almost analytic extension, by \Cref{thm:roumieuextension}. 
	It is easy to see (cf.\ \cite[Proposition 1.1]{FurdosNenningRainer}) that the latter condition is preserved by composition of functions.

	Under these assumptions, a function $f$ defined on an open set $U \subseteq \R^d$
	is of class $\cE^{\{\fM\}}$ if and only if $f\o c$ is of class $\cE^{\{\fM\}}$ for all
	$\cE^{\{\fM\}}$-curves in $U$; see \cite{KMRc,KMRq} and \cite[Theorem 10.7.1]{Schindl14}.

	\begin{theorem}[Non-quasianalytic {$\tensor[_d]{\cE}{^{\{\fM\}}}$}]
		\label{thm:matrixmultdimJoris}
		Let $\fM$ be an R-regular weight matrix of R-moderate growth satisfying \eqref{eq:matrixnonqa}.  
		Then $\tensor[_d]{\cE}{^{\{\fM\}}}$ has property $(\sD)$.
	\end{theorem}
	
	\begin{proof}
		This follows immediately from \Cref{thm:matrix1dimJoris} and the above observations. 
	\end{proof}

	Note that \Cref{thm:matrixmultdimJoris} implies \Cref{thm:ThilliezJoris} as a special case.

	\begin{remark} \label{rem:intersectableF}
		The family $\mathbf Q = \{Q^n\}_{n \in \N_{\ge 1}}$ of quasianalytic intersectable weight sequences referred to at the end of \Cref{rem:intersectable}
		actually is a regular weight matrix of moderate growth. 
		The Roumieu class $\cE^{\{\mathbf Q\}}$ is quasianalytic and, since \Cref{thm:quasiJoris} applies to every $M \in \mathbf Q$, 
		we conclude that $\tensor[_d]{\cE}{^{\{\mathbf Q\}}}$ has property $(\sD)$. 

		Note that there is no weight sequence $M$ with $\cE^{\{M\}} = \cE^{\{\mathbf Q\}}$ and 
		no weight function $\om$ with $\cE^{\{\om\}} = \cE^{\{\mathbf Q\}}$. 
		This follows from the fact that $Q^{n} \le Q^{n+1} \not\preceq Q^{n}$,
		in analogy to the proof given in \cite[Theorem 5.22]{RainerSchindl12}; see also Remark 5.25 there.   
	\end{remark}

	\section{Holomorphic Approximation of functions in \texorpdfstring{$\cE^{[\fM]}$}{EM}} \label{sec:holapprox}
	
	In this section we prove a characterization of the class $\cE^{[\fM]}$ (in dimension one) by holomorphic approximation.
	It generalizes \cite[Proposition 3.3.2]{Thilliez:2020ac}.

	For notational convenience, we set $\|f\|_{A} := \sup_{z \in A} |f(z)|$ for any complex valued function $f$, where $A$ is any set in the domain of $f$.

	\subsection{Some preparatory observations}

	\begin{lemma}
	\label{lem:hmoderategrowth}
		Let $M,N$ be weight sequences satisfying $m_k^{1/k} \to \infty$, $n_k^{1/k} \to \infty$, and $C:=\on{mg}(M,N)<\infty$. Then
		\begin{align}
			\label{eq:hmoderategrowth}
			h_m(t) &\le C^jn_j t^j h_n(Ct), \quad t>0, ~j \in \N,
			\\
			\label{eq:mgsquare}
			h_m(t) &\le h_n\Big(\frac{eC}{2}t\Big)^2, \quad t>0.
		\end{align}
	\end{lemma}
	
	\begin{proof}
		Note that $\on{mg}(m,n) \le \on{mg}(M,N)$.  Thus, for all $j \in \N$ and $t>0$,
		\begin{align*}
		h_m(t) &\le \inf_{k\ge 0} m_{k+j}t^{k+j} \le \inf_{k\ge 0}  n_j n_k (Ct)^{k+j}
		= C^j n_j t^j h_n(Ct).
		\end{align*}
		For \eqref{eq:mgsquare} we refer to \cite[Lemma 3.13]{Rainer:2019ac}.
	\end{proof}

	For $\ep>0$ let $\Om_\ep$ denote the interior of the ellipse in $\C$ 
	with vertices $\pm \cosh(\varepsilon)$ and co-vertices $\pm i\sinh(\varepsilon)$.
	By $\cH(\Om_\ep)$ we denote the space of holomorphic functions on $\Om_\ep$.	
	The following lemma is a simple modification of \cite[Lemma 3.2.4]{Thilliez:2020ac}.
	
	\begin{lemma}
		\label{lem:324v}
		Let $M,N$ be two weight sequences satisfying $m_k^{1/k} \to \infty$, $n_k^{1/k} \to \infty$, 
		and $C:=\on{mg}(M,N)<\infty$. Let $\ep>0$.
		Let $g \in \cH(\Om_\ep)\cap C^0(\ol \Om_\ep)$ and assume that there are constants $L,a_1,a_2>0$ such that
		\[
			\|g\|_{\Om_\ep} \le L,\quad  \|g\|_{[-1,1]} \le a_1h_m(a_2\varepsilon).
		\]
		Then with $a_3:=\max\{a_1,L\}$ and $a_4:=eCa_2$
		we have
		\[
		\|g\|_{\Om_{\ep/2}} \le a_3 h_n(a_4 \varepsilon).
		\]
	\end{lemma}

	\begin{proof}
		Let $f(z):=\frac{1}{a_1} g (\sin(\ve z))$. Since $z \mapsto \sin(\ve z)$ maps 
		the horizontal strip $S:=\{z \in \C:~|\Im(z)|<1\}$ to $\Om_\ve$,
		we get that $f \in \cH(S) \cap C^0(\ol S)$ is bounded by $K:= \max\{1,\frac{L}{a_1}\}$ 
		on the whole of $S$ and by $h_m(a_2\ve)$ on $\R$. Thus an application of Hadamard's three lines theorem gives
		\[
		|f(z)|\le h_m(a_2\ve)^{1-|\Im(z)|}K^{|\Im(z)|}, \quad z \in S.
		\]
		Since $h_m\le 1$ and every $w\in \Om_{\ve/2}$ can be written as $w = \sin(\ve z)$ 
		for some $w \in S$ with $|\Im(w)|\le 1/2$, we obtain
		\[
		|g(w)|\le a_1 (K h_m(a_2\varepsilon))^{1/2}.
		\]
		The statement follows from \eqref{eq:mgsquare}.
	\end{proof}

	\subsection{Condition \texorpdfstring{$(\cP_{[\fM]})$}{PM}}
		Let $\fM$ be a weight matrix.
	\begin{enumerate}
		\item[$(\cP_{\{\fM\}})$] We say that a function $f : [-1,1] \to \C$ satisfies $(\cP_{\{\fM\}})$ if 
	 there exist
	 $M \in \fM$, constants $K,c_1,c_2>0$, and 
	  a family $(f_\ep)_{0<\ep\le \ep_0}$ of functions
	 $f_\ve \in \cH(\Om_\ep) \cap C^0(\ol \Om_\ep)$ such that for all $0<\ep\le \ep_0$,
	\begin{align}
		\label{PM1rmatrix}
		\|f_\varepsilon\|_{\Om_\ep} &\le K,\\
			\label{PM2rmatrix}
			\|f-f_\ep\|_{[-1,1]} &\le c_1 h_m(c_2\varepsilon).
	\end{align}
		\item[$(\cP_{(\fM)})$] We say that a function $f : [-1,1] \to \C$ satisfies $(\cP_{(\fM)})$ if 
	for all $M \in \fM$ and all $c_2 > 0$ 
	there exist constants $K,c_1>0$ and a family $(f_\ep)_{0<\ep\le \ep_0}$ of functions 
	$f_\ve \in \cH(\Om_\ep) \cap C^0(\ol \Om_\ep)$ such that 
	\eqref{PM1rmatrix} and \eqref{PM2rmatrix} hold
	for all $0<\ep\le \ep_0$.
	\end{enumerate}
	Note that $(\cP_{\{\fM\}})$ generalizes condition $(\cP_{M})$ of \cite{Thilliez:2020ac}.

	\subsection{Description by holomorphic approximation}

	\begin{theorem}
		\label{prop:332v}
			\thetag{i} Let $M^{(i)}$,  $1\le i \le 3$, be weight sequences with $(m^{(i)}_k)^{1/k} \rightarrow \infty$ and
			\begin{gather}
				\E B_1\ge 1 \A t>0 : \ol \Ga_{ m^{(2)}} (B_1 t) \le \ul \Ga_{ m^{(1)}}(t), \label{eq:countcomp}
				\\
				\E B_2\ge 1 \A j \in \N : m^{(2)}_{j+1} \le B_2^{j+1}m^{(3)}_j. \label{ass2}
			\end{gather}
			Then for each $f \in \cB^{M^{(1)}}_{B_0}((-1,1))$
			there exist positive constants $K,c_1,c_2$ and functions $f_\ve \in \cH(\Om_\ep)\cap C^0(\ol \Om_\ep)$
			such that for all small $\ep>0$
			\begin{equation}
				\label{eq:holapprox}
				 \|f_\ve\|_{\Om_\ep}\le K, \quad \|f-f_\ve\|_{[-1,1]}\le c_1 h_{m^{(3)}}(c_2\ve).
			\end{equation}
			The constants $K,c_1,c_2$ only depend on $B_i$, in particular, $c_2 = CB_0B_1$, where $C$ is an absolute constant.

			\thetag{ii} Let $N^{(i)}$, $1\le i \le 3$, be weight sequences with
			$(n^{(i)}_k)^{1/k} \rightarrow \infty$ and $\on{mg}(N^{(i)},N^{(i+1)})=D^{(i)}<\infty$.
			Let $f : [-1,1] \to \C$ be a function. 
			Assume that there exist positive constants $K,c_1,c_2$ and functions $f_\ve \in \cH(\Om_\ep)\cap C^0(\ol \Om_\ep)$
			such that for all small $\ep>0$ 
			\begin{equation}
				\label{eq:holapprox2}
				 \|f_\ve\|_{\Om_\ep}\le K, \quad \|f-f_\ve\|_{[-1,1]}\le c_1 h_{n^{(1)}}(c_2\ve).
			\end{equation}
			Then
			$f \in \cB^{N^{(3)}}_\si((-b,b))$ for every $b<1$,
			where $\si := \frac{2eD^{(1)} D^{(2)}c_2}{E(1-b)}$ and $E$ is an absolute constant.
			
			\thetag{iii} If $\fM$ is a [regular] weight matrix of [moderate growth], then
		\begin{equation}
			\label{eq:PMmatrix}
			f \in \cB^{[\fM]}((-1,1)) \Rightarrow f ~\text{satisfies}~(\cP_{[\fM]})\Rightarrow f \in \cE^{[\fM]}((-1,1)).
		\end{equation}
	\end{theorem}
	
	Note that [regularity] of $\fM$ is needed
	for the first implication in (iii),  [moderate growth] for the second.
	Item (iii) generalizes  \cite[Proposition 3.3.2]{Thilliez:2020ac}.

	\begin{proof}	
	We follow closely the proof of \cite[Proposition 3.3.2]{Thilliez:2020ac}.
	
	(i)
	Let $f \in \cB^{M^{(1)}}_{B_0}((-1,1))$. By \cite[Proposition 3.12]{FurdosNenningRainer}, there are constants $c_1,c_2>0$ and a function $F \in C^1_c(\C)$
	extending $f$ such that
	\begin{equation}
	\label{eq:dbext}
	|\db F (z)| \le c_1 h_{m^{(3)}}(c_2 d(z,[-1,1])), \quad z \in \C.
	\end{equation}
	Note that $c_1=c_1(\|f\|^{M^{(1)}}_{B_0}, B_0,B_1,B_2)$ and $c_2=12B_0B_1$.
	Then
	$w_\ve:=  \db F \,\mathbf{1}_{\Om_\ve}$ satisfies
	\[
		\|w_\ep\|_{\C} \le c_1 h_{m^{(3)}}(Cc_2 \ep),
	\]
	where $C>0$ is an absolute constant such that $d(z,[-1,1]) \le C\ep$ for $z \in \Om_\ep$.
	Moreover, the bounded continuous function
	\[
	v_\ve(z):= \frac{1}{2\pi i}\int_{\C} \frac{w_\ve (\ze)}{\ze -z} \, d\ze \wedge d \ol \ze
	\]
	satisfies $\db v_\ve = w_\ve$ in the distributional sense, and we have
	\begin{equation} \label{eq:estvep}
		\|v_\ve\|_{\C} \le c_1 h_{m^{(3)}}(C c_2 \ve).	
	\end{equation}
	So $f_\ve :=F- v_\ve$ is holomorphic on $\Om_\ve$ and continuous on $\ol \Om_\ep$.
	The estimates \eqref{eq:dbext} and \eqref{eq:estvep} easily imply \eqref{eq:holapprox}.

	(ii)
	Let $f : [-1,1] \to \C$ satisfy \eqref{eq:holapprox2}.
	Consider $g_\ep := f_\ve - f_{2\ve} \in \cH(\Om_{\varepsilon}) \cap C^0(\ol \Om_{\varepsilon})$.
	Then $\|g_\ep\|_{\Om_\ep} \le 2 K$  and
	$\|g_\ve\|_{[-1,1]} \le 2c_1 h_{n^{(1)}}(2c_2\ve)$.
	By \Cref{lem:324v},
	\[
		\|g_\ve\|_{\Om_{\varepsilon/2}} \le \max\{c_1,2K\} \, h_{n^{(2)}}(2eD^{(1)}c_2 \ve).
	\]
	There exists a (universal) constant $E>0$ such that for any $b<1$ the closed disk with 
	radius $E(1-b)\ve$ around any $x \in [-b,b]$ is contained in $\Om_{\varepsilon/2}$.
	The Cauchy estimates and \eqref{eq:hmoderategrowth} yield
	\begin{align*}
		\| g_\ve^{(j)}\|_{[-b,b]} &\leq \frac{\max\{c_1,2K\}}{(E(1-b)\ep )^{j}}  j!\,   h_{n^{(2)}}(2eD^{(1)}c_2\varepsilon)	
		\\
		&\leq \max\{c_1,2K\} \Big(\frac{2 eD^{(1)} D^{(2)} c_2}{E(1-b)}\Big)^j  N_j^{(3)}\,   
		h_{n^{(3)}}(2eD^{(1)} D^{(2)}c_2\varepsilon),
	\end{align*}
	which means
	$\| g_{\ep}\|^{N^{(3)}}_{\si,[-b,b]} \le \max\{c_1,2K\}\,  h_{n^{(3)}}(2eD^{(1)} D^{(2)}c_2\varepsilon)$
	for $\si = \frac{2eD^{(1)} D^{(2)}c_2}{E(1-b)}$. Thus, if $\ep_0>0$ is such that \eqref{eq:holapprox2} 
	holds for all $0<\ep\le \ep_0$, then 
	\[
	g := f_{\ve_0} + \sum_{j = 1}^\infty g_{\ve_0 2^{-j}} 
	= f_{\ve_0} + \sum_{j = 1}^\infty (f_{\ve_0 2^{-j}} - f_{\ve_02^{-j+1)}})
	\]
	converges absolutely in the Banach space $\cB^{N^{(3)}}_\si([-b,b])$. Clearly, for every $k \in \N$,
	\[
	g =  f_{\ve_02^{-k}} + \sum_{j = k+1}^\infty (f_{\ve_0 2^{-j}} - f_{\ve_02^{-j+1)}}),
	\]
	and $f = g$ on $[-b,b]$, since for $x \in [-b,b]$,
	\[
	|f(x)-g(x)| \le |f(x)-f_{\ve_02^{-k}}(x)| + \Big|\sum_{j = k+1}^\infty (f_{\ve_0 2^{-j}}(x) - f_{\ve_02^{-j+1)}}(x))\Big|
	\]
	which tends to $0$ as $k \to \infty$, by \eqref{eq:holapprox2} and absolute convergence of the sum.

	(iii)
	For the first implication in \eqref{eq:PMmatrix} in the Roumieu case, observe that for $f \in \cB^{\{\fM\}}((-1,1))$ we have 
	$f \in \cB^{M^{(1)}}_{B_0}((-1,1))$
	for some $B_0>0$ and $M^{(1)}\in \fM$.
	Then R-regularity of $\fM$ implies the existence of $M^{(2)}, M^{(3)} \in \fM$ 
	such that \eqref{eq:countcomp} and \eqref{ass2} are satisfied.
	Thus (i) yields the desired holomorphic approximation.

	In the Beurling case take any weight sequence $M^{(3)} \in \fM$. By B-regularity, we find $M^{(1)},~M^{(2)}$ 
	such that \eqref{eq:countcomp} and \eqref{ass2} are satisfied.
	If $f \in \cB^{(\fM)}((-1,1))$, then $f \in \cB^{M^{(1)}}_{B_0}((-1,1))$ for any $B_0>0$. 
	Again (i) yields the desired holomorphic approximation
	(since $c_2 = C B_0 B_1$).

	The second implication in \eqref{eq:PMmatrix} follows from (ii), since [moderate growth] of $\fM$ 
	yields weight sequences $N^{(i)}$ fulfilling the assumptions of (ii).
	\end{proof}
	
	\section{Proofs}
	\label{sec:proof}

	We are now ready to prove the main results.
	We begin with a technical lemma in which 
	we extract and slightly modify the essential arguments of \cite[Section 4]{Thilliez:2020ac}.
	Its general formulation allows us to readily complete the pending proofs.

	\subsection{A technical lemma}		

	\begin{lemma}
		\label{lem:corelemma}
		Let $j$ be a positive integer.
		Let $M^{(i)}$, $1 \le i \le \lceil\log_2(j(j+1))\rceil+7=:k$, 
		be weight sequences satisfying $(m^{(i)}_\ell)^{1/\ell} \rightarrow \infty$  
		and 
		\begin{gather*}
		\E B\ge 1 \A t>0 : \ol \Ga_{ m^{(2)}} (B t) \le \ul \Ga_{ m^{(1)}}(t),\\
		\on{mg}(M^{(i)},M^{(i+1)}) <\infty, \quad \text{ for }2 \le i \le k-1.
		\end{gather*}
		If $f:[-1,1] \rightarrow \C$ is such that
		$f^j,f^{j+1} \in \cB^{[M^{(1)}]}((-1,1))$,
		then $f \in \cE^{[M^{(k)}]}((-1,1))$.
	\end{lemma}

	\begin{proof}
		Set $g:=f^j$ and $h := f^{j+1}$. 

		Let us begin with the Roumieu case.
		By \Cref{prop:332v}(i), 
		there exist families of holomorphic functions $(g_\ve)$, $(h_\ve)$ approximating $g$, $h$, respectively.
		More precisely, there exist positive constants $K,c_1,c_2$ and functions $g_\ep,h_\ep \in \cH(\Om_\ep) \cap C^0(\ol \Om_\ep)$ 
		such that, for all small $\ve>0$,
		\begin{gather}\label{eq:ghbound}
			\max\{\|g_\ve\|_{\Om_\ep},\|h_\ve\|_{\Om_\ep}\} \le K,
			\\
			\label{eq:jj+1approx}
			\max\{\|g- g_\ve\|_{[-1,1]}, \|h- h_\ve\|_{[-1,1]}\} \le c_1h_{m^{(3)}}(c_2\ve).
		\end{gather}
		Then $g_\ve^{j+1} - h_\ve^j \in \cH(\Om_\ep) \cap C^0(\ol \Om_\ep)$ satisfies
		\begin{align*}
			|g_\ve^{j+1} - h_\ve^j| &\le |g_\ve^{j+1} - f^{j(j+1)}| + |f^{j(j+1)} - h_\ve^j|
			\\
			&\le (j+1) \max\{|g_\ve|, |g|\}^j  |g_\ve - g| + j \max\{|h_\ve|, |h|\}^{j-1}  |h_\ve - h|
			\\
			&\le c_3 h_{m^{(3)}}(c_2 \varepsilon), \quad \text{ on } [-1,1].	
		\end{align*}
	 	Thus \Cref{lem:324v} implies
		\begin{equation}
		\label{eq:epshalf}
		\|h_\varepsilon^{j} - g_\varepsilon^{j+1}\|_{\Om_{\ep/2}} \le c_4 h_{m^{(4)}}(Ce c_2  \varepsilon) =: \de_\ep,
		\end{equation}
		where $C$ is chosen such that $C \ge \on{mg}(M^{(i)},M^{(i+1)})$ for all $2 \le i \le k-1$. 
		(Here and below all constants $c_i$ are independent of $\ep$.)

		Consider the continuous function
		\[
		u_\ve := \ph_{\ve}\frac{\ol g_\ve h_\ve}{\max\{|g_\ve|, r_\ve\}^2}, \quad \text{ with } r_\ve:= \de_\ve^{\frac{1}{j+1}},
		\]
		where $\ph_\ve$ is a smooth function compactly supported in $\Om_\ve$ and $1$ on $\Om_{\varepsilon/2}$.
		It coincides with $h_\ep/g_\ep$ in $\Om_{\ep/2} \cap \{|g_\ep|>r_\ep\}$, but is not holomorphic everywhere near $[-1,1]$.
		By taking $\ep>0$ sufficiently small, we may assume that $\de_\ep \le r_\ep \le 1$.

		Lemmas 4.2.1 to 4.2.4 in \cite{Thilliez:2020ac} (which apply without change to our situation)
		lead to a holomorphic approximation $(f_\ep)$ of $f$ by solving a suitable $\ol\p$-problem.
		Indeed, they show (using \eqref{eq:ghbound}, \eqref{eq:jj+1approx}, \eqref{eq:epshalf} and $h_{m^{(3)}}(t) \le h_{m^{(4)}}(eCt/2)$, by
		\eqref{eq:mgsquare} since $h_{m^{(4)}} \le 1$) that
		\begin{align} \label{eq:uepbound}
			\|u_\ep\|_{\Om_{\ep/2}} &\le (2K)^{1/j},
			\\
			\label{eq:fuepbound}
			\|f- u_\ep\|_{[-1,1]} &\le c_5 r_\ve^{1/j}, 
		\end{align}
		and that the bounded continuous function
		\[
		v_\ve(z) := \frac{1}{2\pi i}\int_{\Om_{\ve/2}} \frac{\ol\p u_\ve(\ze)}{\ze - z} \,d \ze \wedge d\ol \ze,
		\]
		which satisfies $\ol\p v_\ep = \ol\p u_\ep  \mathbf{1}_{\Om_{\ep/2}}$ in the distributional sense in $\C$,
		fulfills
		\begin{equation} \label{eq:vepbound}
						\|v_\ep\|_{\Om_{\ep/2}} \le c_6 \de_\ep^{1/s}
		\end{equation}
		where $s$ is any real number with $s > j(j+1)$ (with $c_6$ depending on $s$).

		Then $f_\ve := u_{2\ve}-v_{2\ve}$ is holomorphic in $\Om_\ep$ and continuous on $\C$.
		By \eqref{eq:uepbound} and \eqref{eq:vepbound},
		$\|f_\ep\|_{\Om_{\ep}}$ is uniformly bounded
		for all small $\ep$, and by  \eqref{eq:fuepbound} and \eqref{eq:vepbound},
		\[
		\|f-f_\ve\|_{[-1,1]} \le c_{7}\de_{2\ve}^{1/s}.
		\]
		Put $s:=2^{k-6}=: 2^\ell$.
		A repeated application of \eqref{eq:mgsquare} gives
		\[
		h_{m^{(4)}}(t)^{1/s}\le h_{m^{(k-2)}}((Ce)^\ell t), \quad t>0.
		\]
		Thus, for all small  $\ep$,
		\begin{align} \label{eq:final}
			\|f-f_\ve\|_{[-1,1]} \le c_{7}\de_{2\ve}^{1/s} &= c_7 \big(c_4 h_{m^{(4)}}(2e Cc_2  \varepsilon)\big)^{1/s} \notag
			\\& \le
		 c_7 c_4^{1/s} h_{m^{(k-2)}}(2 c_2(eC)^{\ell+1}\ve).
		\end{align}
		So \Cref{prop:332v}(ii) implies that $f \in \cE^{\{M^{(k)}\}}((-1,1))$. 
		This ends the proof in the Roumieu case.

		For the Beurling we observe that, by assumption, we find for any (small) $c_2>0$ 
	    approximating sequences $(g_\ve)$, $(h_\ve)$
	    such that \eqref{eq:ghbound} and \eqref{eq:jj+1approx} are satisfied. 
	    Then follow the above proof until the end and notice that thus also 
	    in the final approximation \eqref{eq:final}
	    the constant
	    $2c_2(eC)^{\ell+1}$ gets arbitrarily small as $c_2$ gets small. 
	    Again an application of \Cref{prop:332v} completes the proof.
	\end{proof}

\subsection{Proof of \texorpdfstring{\Cref{thm:matrix1dimJoris} --  {$\tensor[_1]{\cE}{^{[\fM]}}$}}{}}

We may assume that there is a positive integer $j$ such that $g=f^j$, $h=f^{j+1}$ 
are elements of the ring $\tensor[_1]{\cE}{^{[\fM]}}$.
By composing with suitable linear reparameterizations, we may further assume 
that they are represented by elements of $\cB^{[\fM]}((-1,1))$ 
which we denote by the same symbols.

In the Roumieu case, there exists $M^{(1)} \in \fM$ such that $g$, $h$ are contained in $\cB^{\{M^{(1)}\}}((-1,1))$
(by the linear order of $\fM$). By R-regularity and R-moderate growth of $\fM$,
we find sequences $M^{(i)} \in \fM$ satisfying the assumptions of \Cref{lem:corelemma}
which implies that $f \in \cE^{\{M^{(k)}\}}((-1,1))$.

In the Beurling case, we fix an arbitrary $M \in \fM$ and we show that $f \in \cE^{(M)}((-1,1))$.
By B-regularity and B-moderate growth of $\fM$, we now get sequences $M^{(i)} \in \fM$ 
as required in \Cref{lem:corelemma}, where $M^{(k)}=M$. 
By assumption, $g$, $h$ are elements of $\cB^{(M^{(1)})}((-1,1))$. 
Thus \Cref{lem:corelemma} gives $f \in \cE^{(M)}((-1,1))$. \qed

\subsection{Proof of \texorpdfstring{\Cref{thm:weightfunction} -- {$\tensor[_1]{\cE}{^{[\om]}}$}}{}}

This is an immediate corollary of \Cref{thm:matrix1dimJoris} and the discussion in \Cref{thm:BMTmatrix}.\qed

\subsection{Proof of \texorpdfstring{\Cref{thm:functionJoris} -- non-quasianalytic {$\tensor[_d]{\cE}{^{[\om]}}$}}{}}
\label{sec:multi}

We reduce the multidimensional result to the one-dimensional one.

In the Roumieu case $\tensor[_d]{\cE}{^{\{\om\}}}$, \Cref{thm:functionJoris} is a simple corollary of
\Cref{thm:matrixmultdimJoris};
the weight matrix $\fS$ from \Cref{thm:BMTmatrix} clearly satisfies \eqref{eq:matrixnonqa} (since $\om$ is non-quasianalytic). 

The Beurling case $\tensor[_d]{\cE}{^{(\om)}}$ can be reduced to the Roumieu case by means of the following lemma 
(which is an adaptation of \cite[Lemma 13]{Rainer:2020ab}).

\begin{lemma} \label{lem:omreduct}
	Let $\om$ be a non-quasianalytic concave weight function.
	Suppose that $f : [0,\infty) \to [0,\infty)$ is any function satisfying $\om(t) = o(f(t))$ as $t \to \infty$.
	Then there exists a non-quasianalytic concave weight function $\tilde \om$
	satisfying $\om(t) = o(\tilde \om(t))$ and $\tilde \om(t) = o(f(t))$ as $t \to \infty$.
\end{lemma}

\begin{proof}
	It suffices to extract some constructions from the proof of \cite[Lemma 13]{Rainer:2020ab} (to which we refer for details).
	We may assume that $\om$ is of class $C^1$. The condition $\om(t) = o(t)$ as $t \to \infty$
	implies that $\om'(t) \searrow 0$ as $t \to \infty$.

	Note that $\log(t) = o(\om(t))$ and $\om(t) = o(f(t))$ imply $f(t) \to \infty$ as $t \to \infty$.
   We define inductively three sequences $(x_n)$, $(y_n)$, and $(z_n)$ with $x_1=y_1=z_1 =0$, $x_2> 0$, and the following
   properties:
   \begin{gather}
         \label{px0}
         \int_{x_n}^\infty \frac{\om(t)}{1+t^2} dt \le  \frac{1}{n^3},
         \\
         \label{px1}
         x_n > 2  y_{n-1} + n,
         \\
         \label{px4}
         f(t) \ge n^2 \om(t), \quad \text{ for all } t\ge x_n,
         \\
         \label{px3}
         \om(x_n) \ge 2^{n-i} \om(z_i), \quad 1 \le i \le n-1,
         \\
         \label{py1}
         \om'(y_n) = \frac{n-1}{n} \om'(x_n),
         \\
         \label{pz1}
         \om(z_n) = n \om(y_n) - (n-1) \big(\om(x_n) + (y_n-x_n) \om'(x_n)\big).
   \end{gather}
   Concavity of $\om$ guarantees well-definedness of these conditions.
Then
   \[
     \tilde \omega(t) :=
      \begin{cases}
        (n-1)  \big(\om(x_n) + (t-x_n) \om'(x_n) \big) - \sum_{i=1}^{n-2} \om(z_{i+1}) & \text{ if } x_n \le t < y_n,
        \\
        n \om(t) - \sum_{i=1}^{n-1} \om(z_{i+1}) & \text{ if } y_n \le t < x_{n+1},
      \end{cases}
   \]
   defines a non-quasianalytic concave weight function of class $C^1$ satisfying 
   \begin{equation} \label{eq:omtildeom}
   	(n-2) \om(t) \le \tilde \om(t) \le n \om(t), \quad \text{ if } t \in [x_n,x_{n+1}) \text{ and } n \ge 2.
   \end{equation}
   (Non-quasianalyticity follows from \eqref{px0} and the second inequality in \eqref{eq:omtildeom}; cf.\ \cite[Remark 14]{Rainer:2020ab}.)
   Together with \eqref{px4} this implies that
    $\om(t) = o(\tilde \om(t))$ and $\tilde \om(t) = o(f(t))$ as $t \to \infty$.
\end{proof}

Suppose that $g=f^j$, $h=f^{j+1}$ are representatives (of the corresponding germs) 
belonging to $\cB^{(\om)}(U)$ on some relatively compact $0$-neighborhood $U$ in $\R^d$
and consider the sequence $L_k$ defined in \eqref{eq:defL}.
Then for each integer $j \ge 1$ there exists $C_j>1$ such that
\[
	L_k \le C_j \exp(j \vh^*_\om(k/j)), \quad \text{ for all } k \in \N. 	
\]
Defining the function $\ell : [0,\infty) \to \R$ by
\[
	\ell(t) := \log \max\{L_k,1\}, \quad \text{ for } k \le t <k+1,
\]
and performing the subsequent steps in \cite[Section 5]{Rainer:2020ab}, we find
that 
$\ell\le \vh^*_{\tilde \om} + \on{const}$, where $\tilde \om$ is the weight function provided by \Cref{lem:omreduct}.
This means that $g$, $h$ belong to $\cB^{\{\tilde \om\}}(U)$.
Invoking \Cref{thm:functionJoris} in the Roumieu case
shows that $f \in \tensor[_d]{\cE}{^{\{\tilde \om\}}}$.
Since $\om(t) = o(\tilde \om(t))$ as $t \to \infty$
we may conclude that $f \in \tensor[_d]{\cE}{^{(\om)}}$. \qed

\subsection{Proof of \texorpdfstring{\Cref{thm:quasiJoris} -- quasianalytic {$\tensor[_d]{\cE}{^{\{M\}}}$}}{}}
The following lemma is a variant of \cite[Theorem 1.6(3)]{KMRq}.

\begin{lemma} \label{lem:intersectable}
		Let $M$ be a quasianalytic intersectable weight sequence.
		Then:
		\begin{enumerate}[label=\textnormal{(\roman*)}]
			\item $n_k^{1/k} \to \infty$ for all $N \in \cL(M)$.
			\item If $M$ has moderate growth, then for every $N \in \cL(M)$ 
			there exists $N' \in \cL(M)$ such that $\on{mg}(N',N)<\infty$.
		\end{enumerate}
	\end{lemma}

	\begin{proof}
		(i) is obvious, since $m_k^{1/k} \to \infty$ (cf.\ \Cref{rem:intersectable}). 

		(ii) If $M$ has moderate growth, then so has $m$.
		Set $C:= \on{mg}(m,m) <\infty$. For $N \in \cL(M)$ we define $N'$ by
		$n'_k := C^k \min_{0 \le j \le k} n_j n_{k-j}$
		or equivalently
		\[
			n'_{2j} := C^{2j}  \ul\nu_1^2 \ul\nu_2^2 \ul\nu_3^2 \cdots \ul\nu_j^2, \quad
			n'_{2j+1} := C^{2j+1}  \ul\nu_1^2 \ul\nu_2^2 \ul\nu_3^2 \cdots \ul\nu_j^2 \ul\nu_{j+1},
		\]
		where $\ul\nu_k := n_{k}/n_{k-1}$. Then clearly $\on{mg}(N',N)<\infty$.
		Since $\ul\nu_k$ is increasing, so is $\ul\nu'_k := n'_k/n'_{k-1}$, thus $n'$ is log-convex.
		Moreover,
		\[
			n'_{2j} = C^{2j}  n_j^2  \ge \frac{m_{2j}}{m_j^2} n_j^2 \ge m_{2j}
		\]
		and analogously $n'_{2j+1} = C^{2j+1} n_j n_{j+1} \ge m_{2j+1}$, so that  $N'\ge M$.
		It remains to check that $N'$ is non-quasianalytic.
		Since $N'$ is log-convex, the sequence $(N'_k)^{1/k}$ is increasing and 
		so it suffices to show that $\sum_j (N'_{2j})^{-1/(2j)}<\infty$.
		This is clear, since
		\[
			(N'_{2j})^{1/(2j)} = ((2j)!\, C^{2j}  n_j^2)^{1/(2j)} \ge \frac{2C}{e} j n_j^{1/j} \ge \frac{2C}{e} N_j^{1/j}
		\]
		and $N$ is non-quasianalytic.
	\end{proof}

Let $M$ be a quasianalytic intersectable weight sequence of moderate growth. 
Suppose that $g=f^j$, $h=f^{j+1}$ are elements of $\tensor[_d]{\cE}{^{\{M\}}}$.
Since $M$ is intersectable, it suffices to show that $f \in \tensor[_d]{\cE}{^{\{N\}}}$ for every $N \in \cL(M)$.
Fix such $N$.
By \Cref{lem:intersectable}, there exist
$N^{(1)}, \dots, N^{(k)} \in \cL(M)$ with $N^{(k)}=N$ such that the requirements of \Cref{lem:corelemma} are satisfied.
(Note that $\cL(M)$ is not a weight matrix in the sense of \Cref{def:matrix}, because it is not totally ordered.)

Let $U$ be an open $0$-neighborhood in $\R^d$ on which we have  
$g, h \in \cE^{\{M\}}(U)$ for representatives which are denoted by the same symbols.
Take any curve $c \in \cE^{\{N^{(1)}\}}(\R, U)$ with compact support. 
Then, by composition closedness of $\cE^{\{N^{(1)}\}}$ as $n^{(1)}$ is log-convex,
we have $g\circ c, h\circ c \in \cE^{\{N^{(1)}\}}(\R)$.
After a linear change of variables, we may assume that $g\circ c, h\circ c \in \cB^{\{N^{(1)}\}}((-1,1))$.
Thus \Cref{lem:corelemma} yields that $f\circ c \in \cE^{\{N\}}((-1,1))$.
This implies that $f \in \cE^{\{N\}}(U)$, by \cite[Theorem 2.7]{KMRq}.\qed

\subsection*{Acknowledgements} 
We wish to thank Prof.\ Jos\'{e} Bonet Solves 
for drawing our attention to \cite{Thilliez:2020ac}.


\def\cprime{$'$}
\providecommand{\bysame}{\leavevmode\hbox to3em{\hrulefill}\thinspace}
\providecommand{\MR}{\relax\ifhmode\unskip\space\fi MR }
\providecommand{\MRhref}[2]{%
  \href{http://www.ams.org/mathscinet-getitem?mr=#1}{#2}
}
\providecommand{\href}[2]{#2}

\end{document}